\documentclass[reqno, 11pt]{amsart}  

\usepackage{amssymb}
\usepackage{mathrsfs}

\usepackage{graphicx}
\usepackage[usenames,dvipsnames,svgnames]{xcolor}
\usepackage[shortlabels]{enumitem} 

\usepackage[colorlinks, urlcolor=Navy, linkcolor=Navy, citecolor=Navy]{hyperref} 

\usepackage[all]{xy}
\usepackage{tikz}
\usepackage{tikz-cd}
\usetikzlibrary{arrows,shapes,chains,matrix,positioning,scopes,cd,patterns}

\usepackage{todonotes}

\DeclareFontFamily{U}{mathx}{\hyphenchar\font45}
\DeclareFontShape{U}{mathx}{m}{n}{
      <5> <6> <7> <8> <9> <10>
      <10.95> <12> <14.4> <17.28> <20.74> <24.88>
      mathx10
      }{}
\DeclareSymbolFont{mathx}{U}{mathx}{m}{n}
\DeclareFontSubstitution{U}{mathx}{m}{n}
\DeclareMathAccent{\widecheck}{0}{mathx}{"71}
\DeclareMathAccent{\wideparen}{0}{mathx}{"75}

\usepackage[top=30truemm,bottom=30truemm,left=20truemm,right=20truemm]{geometry}

\theoremstyle{plain}
\newtheorem{thm}{Theorem}[section]
\newtheorem{lem}[thm]{Lemma}

\newtheorem{prop}[thm]{Proposition}

\newtheorem{thm*}{Theorem}
\theoremstyle{definition}
\newtheorem{dfn}[thm]{Definition}
\newtheorem{exa}[thm]{Example}
\newtheorem{rem}[thm]{Remark}

\numberwithin{equation}{section}

\DeclareMathOperator{\Hom}{Hom}

\DeclareMathOperator{\End}{End}
\DeclareMathOperator{\Ext}{Ext}

\renewcommand{\mod}{\mathop{\mathsf{mod}}\nolimits}
\DeclareMathOperator{\Mod}{\mathsf{Mod}}

\DeclareMathOperator{\add}{\mathsf{add}}
\DeclareMathOperator{\Gen}{\mathsf{Gen}}
\DeclareMathOperator{\gen}{\mathsf{gen}}
\DeclareMathOperator{\proj}{\mathsf{proj}}
\DeclareMathOperator{\Proj}{\mathsf{Proj}}

\DeclareMathOperator{\Fl}{\mathsf{fl}}

\renewcommand{\Im}{\mathop{\mathrm{Im}}}

\DeclareMathOperator{\projdim}{proj\text{.}dim}

\DeclareMathOperator{\Dim}{\underline{dim}}

\DeclareMathOperator{\rank}{rank}
\DeclareMathOperator{\urank}{\underline{\rank}}

\DeclareMathOperator{\thick}{\mathsf{thick}}

\DeclareMathOperator{\Spec}{Spec}

\newcommand{\supp}{\operatorname{Supp}\nolimits}

\renewcommand{\epsilon}{\varepsilon}

\newcommand{\ep}{\epsilon}

\renewcommand{\L}{\Lambda}

\newcommand{\bN}{\mathbb{N}}

\newcommand{\bT}{\mathbb{T}}

\newcommand{\bZ}{\mathbb{Z}}

\newcommand{\m}{\mathfrak{m}}

\newcommand{\p}{\mathfrak{p}}
\newcommand{\q}{\mathfrak{q}}

\newcommand{\cA}{\mathcal{A}}

\newcommand{\cS}{\mathcal{S}}
\newcommand{\cT}{\mathcal{T}}

\newcommand{\cX}{\mathcal{X}}

\newcommand{\sD}{\mathsf{D}}

\newcommand{\sK}{\mathsf{K}}

\newcommand{\tors}{\mathop{\mathsf{tors}}}

\newcommand{\ftors}{\mathop{\mathsf{f}\text{-}\mathsf{tors}}}

\newcommand{\stors}{\mathop{\mathsf{s}\text{-}\mathsf{tors}}}

\newcommand{\excep}{\mathop{\mathsf{excep}}}

\newcommand{\twopsilt}{\mathop{\mathsf{2}\text{-}\mathsf{psilt}}}
\newcommand{\twosilt}{\mathop{\mathsf{2}\text{-}\mathsf{silt}}}

\newcommand{\stilt}{\mathop{\mathsf{s}\text{-}\mathsf{tilt}}}
\newcommand{\ptilt}{\mathop{\mathsf{ptilt}}}

\newcommand{\Cv}{\mathop{\mathsf{Cv}}}
\newcommand{\Clus}{\mathop{\mathsf{Clus}}}
\newcommand{\pClus}{\mathop{\mathsf{pClus}}\nolimits}

\begin{document}
\title[Schur roots and Tilting modules of acyclic quivers over commutative rings]{Schur roots and Tilting modules of acyclic quivers over commutative rings}
\author[O. Iyama]{Osamu Iyama}
\author[Y. Kimura]{Yuta Kimura}
\address{Osamu Iyama : Graduate School of Mathematical Sciences, The University of Tokyo, 3-8-1 Komaba Meguro-ku Tokyo 153-8914, Japan}
\email{iyama@ms.u-tokyo.ac.jp}
\address{Yuta Kimura : Department of Mechanical Engineering and Informatics, Faculty of Engineering, Hiroshima Institute of Technology, 2-1-1 Miyake, Saeki-ku Hiroshima 731-5143, Japan}
\email{y.kimura.4r@cc.it-hiroshima.ac.jp}
\begin{abstract}
Let $Q$ be a finite acyclic quiver and $\cA_Q$ the cluster algebra of $Q$. It is well-known that for each field $k$, the additive equivalence classes of support tilting $kQ$-modules correspond bijectively with the clusters of $\cA_Q$. The aim of this paper is to generalize this result to any ring indecomposable commutative Noetherian ring $R$, that is, the additive equivalence classes of 2-term silting complexes of $RQ$ correspond bijectively with the clusters of $\cA_Q$.
As an application, for a Dynkin quiver $Q$, we prove that the torsion classes of $\mod RQ$ corresponds bijectively with the order preserving maps from $\Spec R$ to the set of clusters.
\end{abstract}
\maketitle
\section{Introduction}\label{section-introduction}
The class of silting complexes complements the class of tilting complexes from a point of view of mutation, and parametrizes important structures in derived and module categories \cite{KV,AI,BY}.
Tilting and silting theory of \emph{Noetherian $R$-algebras} (that is, algebras over commutative Noetherian rings $R$ which are finitely generated $R$-modules) have been widely studied e.g.\ \cite{Gnedin, IK, Iyama-Wemyss-maximal, Kimura}.

The aim of this paper is to study silting complexes over the path algebra $RQ$ of a finite acyclic quiver $Q$ over a commutative Noetherian ring $R$.
Thanks to Kac's theorem \cite{Kac} and the categorification of cluster algebras \cite{CaKe}, the set of clusters of the cluster algebra $\cA_Q$ of $Q$ \cite{FZ} can be described in terms of the root system. More explicitly, using the notion of real Schur roots and the inductive formula of $E$-invariants due to Schofield and Crawley-Boevey \cite{CB2,Sch}, we introduce the set $\Clus(Q)$ of clusters of $\cA_Q$ equipped with a partial order (see Definition \ref{dfn-cluster}). A \emph{precluster} is a subset of a cluster, and a \emph{positive precluster} is a precluster which does not contain negative simple roots. We denote by $\pClus(Q)$ and $\pClus_+(Q)$ the set of preclusters and positive preclusters of $\cA_Q$ respectively.
We denote by $\twosilt RQ$ (respectively, $\twopsilt RQ$, $\stilt RQ$, $\ptilt RQ$) the set of additive equivalence classes of 2-term silting complexes (respectively, 2-term presilting complexes, support tilting modules, partial tilting modules) of $RQ$.
The following our main result generalizes a bijection due to Ingalls-Thomas \cite{Ingalls-Thomas} for the case $R$ is a field.

\begin{thm}[Theorem \ref{thm-cluster-siltm}]\label{intro-thm-cluster-siltm}
Let $R$ be a commutative Noetherian ring which is ring indecomposable, and $Q$ a finite acyclic quiver. Then we have isomorphisms of posets
\[
\Clus(Q) \simeq \twosilt RQ\simeq\stilt RQ
\]
and bijections
\[\pClus(Q) \simeq \twopsilt RQ\ \mbox{ and }\ \pClus_+(Q) \simeq \ptilt RQ.\]
\end{thm}
Now we explain our second aim.
Let $\L$ be a Noetherian $R$-algebra, and $\mod\L$ the category of finitely generated left $\L$-modules.
A subcategory $\cT$ of $\mod\L$ is called a \emph{torsion class} if it is closed under taking factor modules and extensions.
We denote by $\tors\L$ the set of all torsion classes of $\mod\L$. There are a number of results to give explicit descriptions of $\tors \L$, e.g.\ \cite{AIR, Gabriel, Ingalls-Thomas, IK, IK2, Stanley-Wang}.

When a Noetherian algebra $\L$ is \emph{compatible} (see Section \ref{section 4}), there is a nice description of $\tors\L$, see \eqref{r description}.
For example, if $R$ is a semilocal ring of dimension one, then any Noetherian $R$-algebra is compatible \cite[Theorem 1.5]{IK}.
Another sufficient condition is given by \cite[Corollary 1.8]{IK}.
In particular, if $R$ is a commutative Noetherian ring containing a field and $Q$ is a Dynkin quiver, then the path algebra $RQ$ is compatible.
The following our second result enables us to drop the assumption that $R$ contains a field.
\begin{thm}[Theorem \ref{thm-torsRQ}]
Let $Q$ be a Dynkin quiver, and $R$ a commutative Noetherian ring.
Then $RQ$ is compatible, and we have an isomorphism of posets
\[
\tors RQ \simeq \Hom_{\rm poset}(\Spec R, \Clus(Q)).
\]
\end{thm}
Here, for posets $X, Y$, we denote by $\Hom_{\rm poset}(X, Y)$ the set of order-preserving maps from $X$ to $Y$. This is a poset with the following order: $g \leq f$ if $g(x)\leq f(x)$ holds for any $x\in X$.

\medskip
\noindent{\bf Conventions.} 
Unless stated otherwise, all modules are left modules. The composition of morphisms $f:X\to Y$ and $g:Y\to Z$ is denoted by $gf:X\to Z$.

For an additive category $\cA$ and an object $X\in\cA$, we denote by $\add X$ the full subcategory of $\cA$ consisting of direct summands of finite direct sums of $X$. We call $X, Y\in \cA$ \emph{additively equivalent} if $\add X = \add Y$.
\section{Preliminaries}

\subsection{Preliminaries on Schur roots}

Throughout this subsection, let $Q$ a finite acyclic quiver. 
Define bilinear forms $\langle-,-\rangle$ and $(-,-):\bZ^{Q_0}\times\bZ^{Q_0}\to\bZ$ by
\begin{align*}
\langle x,y\rangle:=\sum_{i\in Q_0}x_iy_i-\sum_{a\in Q_1}x_{s(a)}y_{t(a)}\ \mbox{ and }\ (x,y):=\langle x,y\rangle+\langle y,x\rangle.
\end{align*}
We denote by $\ep_i=(0,\dots,0, 1, 0, \dots, 0)\in\bZ^{Q_0}$ the $i$-th simple root, and $\Pi:=\{\ep_i \mid i \in Q_0\}$.
For each $i\in Q_0$, a simple reflection $s_i\in \mathrm{Aut}_{\bZ}(\bZ^{Q_0})$ is defined by
\[s_i(x):=x-(x,\ep_i)\ep_i.\]
The \emph{Coxeter group} $W_Q$ of $Q$ is a subgroup of $\mathrm{Aut}_{\bZ}(\bZ^{Q_0})$ generated by $\{s_i \mid i \in Q_0\}$.
Consider the set of \emph{real roots:}
\[\Phi^{\rm re}:=W_Q\Pi = \{w\ep_i\mid w\in W_Q,\ i\in Q_0\}.
\]

Let $k$ be an algebraically closed field. A $kQ$-module $X$ is said to be \emph{exceptional} if $\End_{kQ}(X)=k$ and $\Ext_{kQ}^1(X,X)=0$.
We denote by $\excep kQ$ the set of isomorphism classes of exceptional $kQ$-modules.
As a special case of Kac's theorem \cite{Kac}, we have an injective map $\Dim : \excep kQ \to \Phi^{\rm re}$.
We define the set of \emph{real Schur roots} as
\[
\Phi^{\rm rS} : = \Dim(\excep kQ).
\]
Note that $\Phi^{\rm rS}$ does not depend on the choice of $k$ (e.g.\ \cite[Corollary 4.8]{HK}).

For $x,y\in \bZ^{Q_0}$, we write $x \leq y$ if $x_i  \leq y_i$ for each $i\in Q_0$, and write $x<y$ if $x \leq y$ and $x \neq y$ hold.
\begin{dfn}\label{dfn-E-inv}\cite{Sch}
Define a map $E(-,-):\bN^{Q_0}\times\bN^{Q_0}\to\bN$ inductively by the formula $E(x,0)=0=E(0,y)$ and
\[E(x,y):=\max\{-\langle x', y-y'\rangle\mid0\le x'\le x,\ 0\le y'\le y,\ E(x',x-x')=0=E(y',y-y')\}.\]
\end{dfn}

By \cite[Theorem 5.4]{Sch} (see also \cite[Theorem 10.12.7]{DW}), the following equalities also hold.
\begin{align*}
E(x,y)&=\max\{-\langle x, y-y'\rangle\mid 0\le y'\le y,\ E(y',y-y')=0\}\\
&=\max\{-\langle x-x',y\rangle\mid0\le x'\le x,\ E(x',x-x')=0\}.
\end{align*}
By \cite[Theorem 5.4]{Sch} and \cite{CB2}, for each algebraically closed field $k$, the following equality holds.
\[E(\alpha, \beta)=\min\{\dim_k \Ext_{kQ}^1(M, N) \mid M, N\in\mod kQ,\, \Dim M=\alpha,\, \Dim N=\beta \}. \]

Now we recall results by Crawley-Boevery.
Let $R$ be a commutative ring, and $RQ$ the path algebra.
We say that $X\in\mod RQ$ is a \emph{lattice} if $X$ is a finitely generated projective $R$-module.
For an $RQ$-lattice $X$ such that $e_iX$ is a free $R$-module of finite rank for each $i\in Q_0$ (e.g.\ $R$ is a principal ideal domain), let $\urank X:=(\rank e_i X)_i \in\bN^{Q_0}$.
An $RQ$-module $X$ is \emph{rigid} if $\Ext_{RQ}^1(X, X)=0$, and is \emph{exceptional} if it is a rigid lattice and a natural morphism $R\to \End_{RQ}(X)$ is an isomorphism.
Let $\excep RQ$ be the set of isomorphism classes of exceptional $RQ$-lattices.

\begin{prop}\label{prop-CB-thm1}
Let $Q$ be a finite acyclic quiver and $R$ a principal ideal domain.
\begin{enumerate}[\rm(a)]
\item \cite[Lemma 1]{CB} For each rigid $RQ$-lattices $X$ and $Y$, $\Ext^1_{RQ}(X,Y)$ is free $R$-module of rank $E(\urank X,\urank Y)$.
\item \cite[Theorem 1]{CB} A map $X \mapsto \urank X$ gives a bijection
$\excep R Q\simeq\Phi^{\rm rS}$.
\item \cite[Theorem 2]{CB} Any finitely generated rigid $RQ$-module is a direct sum of exceptional $RQ$-lattices.
\item \cite[Lemma 3.1]{CB3} Let $S$ be a commutative ring, $X$ an $SQ$-lattice and $Y$ an $SQ$-module. Then $\projdim_{SQ}X\leq 1$ holds. For each morphism $S\to T$ of commutative rings, $T\otimes_S\Ext^i_{SQ}(X,Y)\simeq\Ext^i_{TQ}(T\otimes_SX,T\otimes_SY)$ holds for $i=0,1$.
\end{enumerate}
\end{prop}
\subsection{Preliminaries on silting theory}
Let $\L$ be a ring. We denote by $\sK^{\rm b}(\proj \L)$ the homotopy category of the category $\proj \L$ of finitely generated projective $\L$-modules.

\begin{dfn}
Let $\L$ be a ring.
\begin{enumerate}[{\rm (1)}]
\item We say that $P\in \sK^{\rm b}(\proj \L)$ is \emph{presilting} if $\Hom(P, P[i])=0$ holds for any $i>0$, and $P$ is \emph{silting} if it is presilting and the smallest thick subcategory $\thick P$ of $\sK^{\rm b}(\proj \L)$ containing $P$ is $\sK^{\rm b}(\proj \L)$.
\item We say that a complex in $\sK^{\rm b}(\proj \L)$ is \emph{2-term} if it is concentrated in degree $-1$ and $0$.
\item A \emph{2-term silting} (respectively, \emph{2-term presilting}) \emph{complex}  is a silting (respectively, presilting) complex which is 2-term.
We denote by $\twosilt\L$ (respectively, $\twopsilt\L$) the set of additive equivalence classes of 2-term silting (respectively, 2-term presilting) complexes.
Then $\twosilt\L$ is a poset with the following order: $Q \leq P$ if $\Hom_{\sK^{\rm b}(\proj \L)}(P, Q[1])=0$.
\end{enumerate}
\end{dfn}

We denote by $\sD(\L):=\sD(\Mod\L)$ the derived category of the category $\Mod\L$ of left $\L$-modules.
For $M\in\mod\L$, we denote by $\Gen M$ the subcategory of $\Mod\L$ consisting of factor modules of (possibly infinite) direct sums of copies of $M$. Let $\gen M:=\Gen M \cap \mod \L$.

The following observation is standard for finite dimensional algebras \cite{AIR}, \cite[Theorem 3.3]{IJY}.

\begin{lem}\label{idempotent reduction}
Let $\L$ be a ring, $e\in\L$ an idempotent, $\overline{\L}:=\L/\langle e \rangle$ and $\overline{(-)}:=\overline{\L}\otimes_\L-:\sK^{\rm b}(\proj\L)\to\sK^{\rm b}(\proj\overline{\L})$.
For a 2-term complex $P\in\sK^{\rm b}(\proj\L)$, the following conditions are equivalent.
\begin{enumerate}[{\rm(i)}]
\item $\L e[1]\oplus P\in\twosilt\L$ (respectively, $\twopsilt\L$).
\item $eH^0(P)=0$ and $\overline{P}\in\twosilt\overline{\L}$ (respectively, $\twopsilt\overline{\L}$).
\end{enumerate}
\end{lem}

\begin{proof}
For a 2-term complex $P\in\sK^{\rm b}(\proj\L)$, let $Q:=\L e[1]\oplus P$.
Consider subcategories
\begin{align*}
\cT_Q:=\{X\in\Mod\L\mid\Hom_{\sD(\L)}(Q,X[1])=0\}\ \mbox{ and }\ 
\cT_{\overline{P}}:=\{Y\in\Mod\overline{\L}\mid\Hom_{\sD(\overline{\L})}(\overline{P},Y[1])=0\}.
\end{align*}
By \cite[Proposition A.2]{IK}, we have $Q\in\twopsilt\L\Leftrightarrow\cT_Q\supseteq\Gen H^0(Q)$ and $\overline{P}\in\twopsilt\overline{\L}\Leftrightarrow\cT_{\overline{P}}\supseteq\Gen H^0(\overline{P})$.
By \cite[Corollary A.5]{IK}, we have $Q\in\twosilt\L\Leftrightarrow\cT_Q=\Gen H^0(Q)$ and $\overline{P}\in\twosilt\overline{\L}\Leftrightarrow\cT_{\overline{P}}=\Gen H^0(\overline{P})$. 
Consequently, it suffices to show that $\cT_Q=\cT_{\overline{P}}$ and $\Gen H^0(Q)=\Gen H^0(\overline{P})$.
The first equality follows from $\{X\in\Mod\L\mid\Hom_{\sD(\L)}(\L e[1],X[1])=0\}=\Mod\overline{\L}$ and $\Hom_{\sD(\L)}(P,Y[1])\simeq\Hom_{\sD(\overline{\L})}(\overline{P},Y[1])$ holds for each $Y\in\Mod\overline{\L}$, and the second one follows from $H^0(Q)=H^0(P)=H^0(\overline{P})$.
\end{proof}

\begin{prop}\label{idempotent surjective}
Let $\L$ be a left Noetherian ring, $e\in\L$ an idempotent and $\overline{\L}:=\L/\langle e \rangle$ and $\overline{(-)}:=\overline{\L}\otimes_\L-:\sK^{\rm b}(\proj\L)\to\sK^{\rm b}(\proj\overline{\L})$.
If $\overline{(-)}:\proj\L\to\proj\overline{\L}$ is dense, then we have surjective maps
\[\overline{(-)}:\twosilt\L\to\twosilt\overline{\L}\ \mbox{ and }\ \overline{(-)}:\twopsilt\L\to\twopsilt\overline{\L},\]
where the first map is order-preserving. Moreover, for each $Q\in\twosilt\overline{\L}$ (respectively, $Q\in\twopsilt\overline{\L}$), there exists $P\in\twosilt\L$ (respectively, $P\in\twopsilt\L$) satisfying $Q\simeq\overline{P}$ and $H^0(Q)\simeq H^0(P)$.
\end{prop}

\begin{proof}
These maps are well-defined by \cite[Lemma 2.13]{IK}. Since the functor $\overline{(-)}:\proj\L\to\proj\overline{\L}$ is full and dense, for any 2-term complex $Q\in\sK^{\rm b}(\proj\overline{\L})$, there exists $S=(P^{-1}\xrightarrow{f}P^0)\in\sK^{\rm b}(\proj\L)$ such that $Q\simeq\overline{S}$.
Since $\L$ is left Noetherian, there exists an exact sequence $P'\xrightarrow{g} P^0\to\overline{P^0}\to0$ with $P'\in\add\L e$. Then $P:=\L e[1]\oplus(P^{-1}\oplus P'\xrightarrow{(f\ g)}P^0)$ satisfies $Q\simeq\overline{S}\simeq\overline{P}$ and $H^0(Q)\simeq H^0(\overline{S})\simeq H^0(\overline{P})$.
If $Q\in\twosilt\overline{\L}$ (respectively, $\twopsilt\overline{\L}$), then $P\in\twosilt\L$ (respectively, $\twopsilt\L$) holds by Lemma \ref{idempotent reduction}. Thus the assertion holds.
\end{proof}

We give an example of $(\L,e)$ such that the assumption in Proposition \ref{idempotent surjective} is satisfied.

\begin{exa}\label{RQ dense}
Let $R$ be a commutative ring, $Q$ an acyclic quiver. For a subset $I\subset Q_0$, let $Q'$ be a full subquiver of $Q$ with $Q'_0=Q_0\setminus I$. Then the functor $RQ'\otimes_{RQ}-:\proj RQ\to\proj RQ'$ is dense.
\end{exa}

\begin{proof}
Let $e:=\sum_{i\in I}e_i$, $f:=1-e$, $\L:=f(RQ)f$ and $J:=f\langle e\rangle f$. Then $RQ'\simeq RQ/\langle e\rangle\simeq \L/J$ holds. Since $\proj\L\simeq\add RQf\subset\proj RQ$, it suffices to show that $RQ'\otimes_\L-:\proj\L\to \proj RQ'$ is dense. This follows from the fact that $J$ is a nilpotent ideal of $\L$.
\end{proof}

In the rest, let $R$ be a commutative Noetherian ring and $\L$ a Noetherian $R$-algebra.
We use the following observation later, where $(-)_\p$ is the localization at $\p\in\Spec R$.
\begin{prop}\cite[Proposition 4.15]{IK}\label{prop-local-iso}
Assume that $R$ is ring indecomposable.
Then for a Noetherian $R$-algebra $\L$, the following {\rm (a)} and {\rm (a$'$)} are equivalent.
\begin{enumerate}
	\item[{\rm (a)}] $(-)_\p : \twosilt\L \to \twosilt \L_\p$ is an isomorphism of posets for any $\p\in\Spec R$.
	\item[{\rm (a$'$)}] The following statements hold.
		\begin{enumerate}[{\rm (i)}]
		\item $(-)_\p : \twosilt \L \to \twosilt \L_\p$ is surjective for any $\p\in\Spec R$.
		\item $(-)_\q : \twosilt \L_\p \to \twosilt \L_\q $ is an isomorphism of posets for any pair $\p \supseteq \q$ in $\Spec R$.
		\end{enumerate}
\end{enumerate}
\end{prop}
We denote by $\Fl\L$ the category of finite length $\L$-modules.
For $\p\in\Spec R$, we have an isomorphism of posets
\begin{equation}\label{torsfl}
\tors(\Fl\L_\p) \to \tors(\kappa(\p)\otimes_{R}\L),\, \cT \mapsto \cT \cap \mod (\kappa(\p)\otimes_{R}\L),
\end{equation}
see \cite[Theorem 5.4]{Kimura}.
We use the following observation later.
\begin{prop}\cite[Theorem 4.13]{IK}\label{thm-r-isom}
Let $\L$ be a Noetherian $R$-algebra such that $R$ is ring indecomposable and the following conditions are satisfied.
\begin{itemize}
	\item[{\rm (a)}] $(-)_{\p} : \twosilt \L \to \twosilt \L_{\p}$ is an isomorphism of posets for any $\p\in\Spec R$.
	\item[{\rm (b)}] The map $\twosilt \L_{\p} \to \tors(\kappa(\p)\otimes_{R_\p}\L), \,P \mapsto \gen(\kappa(\p)\otimes_{R_\p}H^0(P))$ is surjective for any $\p\in\Spec R$.
\end{itemize}
Then $\L$ is compatible, and we have an isomorphism of posets
\[
\tors\L \simeq \Hom_{\rm poset}(\Spec R, \twosilt\L).
\]
\end{prop}

\begin{proof}
Let $\stors(\Fl\L_\p):=\{\Fl\L_\p \cap \gen H^0(P) \mid P \in \twosilt\L_\p\}$. Thanks to the bijection \eqref{torsfl}, the condition (b) is equivalent to $\stors(\Fl\L_\p)=\tors(\Fl\L_\p)$.
Thus the assertion follows immediately from \cite[Theorem 4.13]{IK}.
\end{proof}

\section{Clusters and silting modules over $RQ$}

We define clusters which play an important role in this paper.

\begin{dfn}\label{dfn-cluster}
Let $Q$ be a finite acyclic quiver, $\Phi^{\rm rS}$ the set of real Schur roots, and $-\Pi = \{-\ep_i \mid i \in Q_0\}$.
\begin{enumerate}[{\rm(a)}]
\item Consider the set of  \emph{cluster variables}
\[\Cv(Q) :=\Phi^{\rm rS}\sqcup (-\Pi).\]
\item For a subset $\cS$ of $\Cv(Q)$, let $\cS_{+}:=\cS \cap \Phi^{\rm rS}$ and $\cS_{-}:=\cS\setminus\cS_{+}$. 
We call $\cS$ a \emph{precluster} if the following conditions are satisfied.
\begin{enumerate}[$\bullet$]
\item[(i)] For each $\alpha,\beta\in\cS_{+}$, we have $E(\alpha,\beta)=0$,
\item[(ii)] For each $\alpha\in\cS_{+}$ and $-\ep_i\in\cS_{-}$, we have $i\notin\supp\alpha:=\{i\in Q_0 \mid \alpha_i\neq 0\}$.
\end{enumerate}
A \emph{cluster} is a precluster set which is maximal with respect to the inclusion. A precluster $\cS$ is called \emph{positive} if $\cS=\cS_+$.
We denote by $\Clus(Q)$ (respectively, $\pClus(Q)$, $\pClus_+(Q)$) the set of all clusters (respectively, preclusters, positive preclusters).
\item For $\cS, \cS' \in \Clus(Q)$, we write $\cS \geq \cS'$ if the following conditions are satisfied.
\begin{enumerate}[{\rm(i)}]
\item $E(\alpha, \alpha')=0$ for each $\alpha\in \cS$ and $\alpha'\in \cS'$.
\item $\cS_-\subset\cS'_-$.
\end{enumerate}
\end{enumerate}
\end{dfn}

\begin{rem}\label{rem-idem}
For a full subquiver $Q'$ of $Q$, $RQ'$ is a factor algebra of $RQ$ modulo the ideal generated by the idempotent $e=\sum_{i\in Q_0\setminus Q_0'}e_i$.
Thus we can identify $\mod RQ'$ with the Serre subcategory of $\mod RQ$ consisting of $M$ satisfying $eM=0$.
In particular, we have
\begin{align*}
\Ext_{RQ}^1(X,Y) \simeq \Ext_{RQ'}^1(X, Y)\quad \mbox{for $X,Y\in\mod RQ'$.}
\end{align*}
For example, an $RQ$-module $M_{\alpha}^R$ is an $RQ'$-module if and only if $\supp \alpha\subseteq Q_0'$ holds.
\end{rem}

\begin{dfn}\label{dfn-sttilt}
Let $R$ be a commutative Noetherian ring, $Q$ a finite acyclic quiver and $T\in\mod RQ$.
\begin{enumerate}[{\rm (1)}]
\item $T$ is a \emph{partial tilting} module if there exists an exact sequence $0\to P^{-1}\to P^0\to T\to0$ with $P^i\in\proj RQ$ and $\Ext^1_{RQ}(T,T)=0$ holds. Also $T$ is a \emph{tilting} module if it is a partial tilting and there exists an exact sequence $0\to RQ\to T^0\to T^1\to0$ with $T^i\in\add T$.
We denote by $\ptilt RQ$ the set of additive equivalence classes of partial tilting $RQ$-modules.
\item $T$ is a \emph{support tilting $RQ$-module} if there exists a full subquiver $Q'$ of $Q$ such that $T$ is a tilting $RQ'$-module. We denote by $\stilt RQ$ the set of additive equivalence classes of support tilting $RQ$-modules.
Then $\stilt RQ$ is a poset with the following order: $N \leq M$ if $\gen N \subseteq \gen M$.
\end{enumerate}
\end{dfn}

We denote the inverse of the map in Proposition \ref{prop-CB-thm1}(b) by
\[M_- : \Phi^{\rm rS}\simeq\excep \bZ Q,\quad \alpha \mapsto M_{\alpha},\]
and let $M_{-\epsilon_i}:=0$ for each $i\in Q_0$.
Let $R$ be a commutative ring, $\alpha\in\Cv(Q)$ and $\cS\subset\Cv(Q)$. We define
\[
M_\alpha^R:= R\otimes_{\bZ}M_{\alpha}\in\mod RQ\ \mbox{ and }\ M_{\cS}^R:=\bigoplus_{\alpha\in\cS}M_{\alpha}^R\in\Mod RQ.
\]
For each $\alpha\in\Phi^{\rm rS}$, we have $\projdim_{\bZ Q}M_\alpha\le 1$ by Proposition \ref{prop-CB-thm1}(d). Let $P_{\alpha}=(P^{-1}_\alpha\to P^0_\alpha)\in \sK^{\rm b}(\proj \bZ Q)$ be a projective resolution of $M_\alpha$. For $i\in Q_0$, let $P_{-\epsilon_i}:=(\bZ Q)e_i[1]\in\sK^{\rm b}(\proj \bZ Q)$. We define
\[P_\alpha^R:= R\otimes_{\bZ}P_{\alpha}\in\sK^{\rm b}(\proj RQ)\ \mbox{ and }\ P_{\cS}^R:=\bigoplus_{\alpha\in\cS}P_{\alpha}^R\in\sK^{\rm b}(\Proj RQ).\]
Clearly $P_\alpha^R\simeq M_\alpha^R$ and $P_\cS^R\simeq M_\cS^R$ hold in $\sD(RQ)$, and $H^0(P_\alpha^R)=M_\alpha^R$ and $H^0(P_\cS^R)=M_\cS^R$ hold in $\Mod RQ$.

Now we are ready to state our main result of this paper.
\begin{thm}\label{thm-cluster-siltm}
Let $R$ be a commutative Noetherian ring which is ring indecomposable, and $Q$ a finite acyclic quiver.
\begin{enumerate}[\rm(a)]
\item We have isomorphisms of posets
\[
M_{-}^R:\Clus(Q) \xrightarrow[\sim]{\ P_{-}^R\ } \twosilt RQ\xrightarrow[\sim]{\ H^0\ }\stilt RQ.\]
\item We have bijections
\[P_{-}^R : \pClus(Q) \simeq \twopsilt RQ\ \mbox{ and }\ M_{-}^R:\pClus_+(Q) \simeq \ptilt RQ.\]
\end{enumerate}
\end{thm}

The first step of the proof of Theorem \ref{thm-cluster-siltm} is the following observation.
\begin{prop}\label{prop-cluster-stilt}
Let $R$ be a commutative Noetherian ring, $Q$ a finite acyclic quiver.
We have the following three injective maps, which are bijective if $R$ is a principal ideal domain.
\begin{enumerate}[\rm(a)]
\item {\rm (cf.\,\cite[Theorem B]{CB3})} $M^R_-:\Phi^{\rm rS}  \to \excep RQ$,
\item $M^R_-:\pClus_+(Q)\to \ptilt RQ$, $P^R_-:\pClus(Q)\to \twopsilt RQ$, $M^R_-:\Clus(Q)\to \stilt RQ$ and $P^R_-:\Clus(Q)\to \twosilt RQ$.
\end{enumerate}
\end{prop}
We start with proving (a).
\begin{proof}[Proof of Proposition \ref{prop-cluster-stilt}(a)]
Since $M_\alpha$ is an exceptional $\bZ Q$-lattice by Proposition \ref{prop-CB-thm1}(b), $M_\alpha^R=R\otimes_{\bZ}M_\alpha$ is an exceptional $RQ$-lattice by Proposition \ref{prop-CB-thm1}(d). Thus the map in (a) is well-defined.
This map is injective since $\urank M_{\alpha}^R=\alpha$ holds.
This map is bijective by Proposition \ref{prop-CB-thm1}(b) if $R$ is a principal ideal domain.
\end{proof}

For $\cS \subset \Cv(Q)$, let $Q^{\cS}$ be the full subquiver of $Q$ with $Q^{\cS}_0=\{i\in Q_0\mid -\ep_i\notin\cS\}$.
\begin{lem}\label{lem-comp-ptilt-a}
Let $R$ be a commutative Noetherian  ring, $Q$ a finite acyclic quiver and $\cS \subset \Cv(Q)$.
Then $\cS\in\pClus(Q)$ if and only if $M_{\cS}^R$ is a partial tilting $RQ^{\cS}$-module if and only if $P_{\cS}^R\in\twopsilt RQ$.
In this case, $|\cS|\leq |Q_0|$.
\end{lem}
\begin{proof}
(i) We prove $\cS\in\pClus(Q)$ if and only if $M_{\cS}^R$ is an $RQ$-module satisfying $\Ext_{RQ}^1(M_{\cS}^R, M_{\cS}^R)=0$.
By Remark \ref{rem-idem}, the condition Definition \ref{dfn-cluster}(ii) is equivalent to that $M_{\cS}^R$ is an $RQ^{\cS}$-module.
Now assume that these conditions hold.
By Proposition \ref{prop-CB-thm1}(d), $M_{\cS}^R$ has projective dimension at most one as both an $RQ$-module and an $RQ^{\cS}$-module.
Moreover, for $\alpha, \beta\in\Phi^{\rm rS}$, by Proposition \ref{prop-CB-thm1}(a)(d), we have
$E(\alpha,\beta)=0$ if and only if $\Ext^1_{\bZ Q}(M_{\alpha} ,M_{\beta})=0$ if and only if $\Ext^1_{RQ}(M_{\alpha}^R ,M_{\beta}^R)=0$.
Thus the condition Definition \ref{dfn-cluster}(i) equivalent to that $\Ext^1_{RQ}(M_{\cS}^R ,M_{\cS}^R)=0$.

(ii) We prove $\cS\in\pClus(Q)$ implies $|\cS|\leq |Q_0|$.
Let $k$ be a field. By (i), $M^k_{\cS}$ is a $kQ$-module with $\Ext_{kQ}^1(M^k_{\cS}, M^k_{\cS})=0$. Thus $|\cS_+|=|M_{\cS_+}^k| \leq |Q^{\cS}_0| = |Q_0|-|\cS_-|$ and hence $|\cS|\leq |Q_0|$ as desired.

(iii) By (i) and (ii), the first equivalence and the last claim are proved.

(iv) We prove the second equivalence. We always have $\Ext^1_{RQ}(M_{\cS}^R ,M_{\cS}^R)\simeq \Hom_{\sD(RQ)}(P_{\cS_+}^R,P_{\cS_+}^R[1])$ and $\Hom_{\sD(RQ)}(P_{\cS}^R,P_{\cS_-}^R[1])=0$. Moreover, $M_{\cS}^R$ is an $RQ^{\cS}$-module if and only if $\Hom_{\sD(RQ)}(P_{\cS_-}^R,P_{\cS_+}^R[1])=0$. Thus the assertion follows. (Alternatively it follows from Lemma \ref{idempotent reduction}.)
\end{proof}
We use the following lemma to prove Proposition \ref{prop-cluster-stilt}(b).
\begin{lem}\label{lem-comp-ptilt}
Let $Q$ be a finite acyclic quiver, $R$ a principal ideal domain, and $\cS \in \pClus(Q)$.
Then the following conditions are equivalent.
\begin{enumerate}[{\rm(i)}]
\item $\cS\in\Clus(Q)$.
\item[{\rm (ii)$_R$}] $M_{\cS}^R$ is a tilting $R Q^{\cS}$-module.
\item[{\rm (iii)}] $|\cS|=|Q_0|$.
\end{enumerate}
\end{lem}
\begin{proof}
Take a morphism of rings $R \to k$ for a field $k$.
By Lemma \ref{lem-comp-ptilt-a}, $M_{\cS}^R$ is a partial tilting $RQ^{\cS}$-module, and $M_{\cS}^k$ is a partial tilting $kQ$-module.

(i) $\Rightarrow$ (ii)$_R$
Taking the Bongartz completion (cf.\,proof of \cite[VI.\,Lemma 2.4]{ASS}), we obtain an $RQ^{\cS}$-module $N$ and an exact sequence
\[
0 \to RQ^{\cS} \to N \to M' \to 0,
\]
with $M'\in\add M_{\alpha}^R$ such that $M_{\cS}^R\oplus N$ is a tilting $RQ^{\cS}$-module.
By this sequence, $N$ is an $RQ^{\cS}$-lattice.
By Proposition \ref{prop-CB-thm1}(b)(c), there exists a subset $\cS'\subset\Phi^{\rm rS}$ such that $\add N=\add M_{\cS'}^R$.
By Lemma \ref{lem-comp-ptilt-a}, $\cS\cup\cS'$ is a precluster.
By maximality of $\cS$, we have $\cS'\subset\cS$ and hence $\add M_{\cS}^R=\add(M_{\cS}^R\oplus N)$.
Thus $M_{\cS}^R$ is a tilting $RQ^{\cS}$-module.

(ii)$_{R}$ $\Rightarrow$ (ii)$_k$
Applying $k\otimes_R-$ to the exact sequence $0 \to RQ^{\cS} \to T^0 \to T^1 \to 0$ with $T^i \in\add M_{\cS}^R$, we obtain an exact sequence $0 \to kQ^{\cS} \to k\otimes_R T^0 \to k\otimes_R T^1 \to 0$ with $k\otimes_RT^i \in\add M_{\cS}^k$.
Thus $M_{\cS}^k$ is a tilting $kQ^{\cS}$-module.

(ii)$_k$ $\Leftrightarrow$ (iii)
We have $|\cS_+|=|M_{\cS}^k|$.
By \cite[VI.\,Corollary 4.4]{ASS}, $M_{\cS}^k$ is a tilting $kQ^{\cS}$-module if and only if $|\cS_+|=|Q_0^{\cS}|$ if and only if $|\cS| = |Q_0|$.

(iii) $\Rightarrow$ (i) 
Let $\cS\subseteq\cS'\in\pClus(Q)$. Then $|Q_0|=|\cS|\leq |\cS'| \leq |Q_0|$ holds by Lemma \ref{lem-comp-ptilt-a}. Thus $\cS=\cS'$.
\end{proof}

\begin{proof}[Proof of Proposition \ref{prop-cluster-stilt}(b)]
By Lemma \ref{lem-comp-ptilt-a}, $M_{-}^R : \pClus(Q) \to \ptilt RQ$ and $P_{-}^R : \pClus(Q) \to \twopsilt RQ$ are well-defined. To prove the well-definedness of the other maps, let $\cS\in\Clus(Q)$.
Applying Lemma \ref{lem-comp-ptilt} for $R=\bZ$, there exists an exact sequence $0 \to \bZ Q^{\cS} \to T^0 \to T^1 \to 0$ with $T^i\in \add M_{\cS}$.
Applying the functor $R\otimes_{\bZ}(-)$, we obtain an exact sequence $0 \to RQ^{\cS} \to R\otimes_{\bZ}T^0 \to R\otimes_{\bZ}T^1 \to 0$ with $R\otimes_{\bZ}T^i \in \add M_{\cS}^R$.
Thus $M_{\cS}^R$ is a tilting $RQ^{\cS}$-module, and hence $M_{\cS}^R\in\stilt RQ$.

To prove $P_{\cS}^R\in\twosilt RQ$, it suffices to show $RQ \in \thick P_{\cS}^R$. Let $e:=\sum_{-\ep_i \in\cS}e_i$, then $P_{\cS}^R=P_{\cS_+}^R \oplus RQ e[1]$. Since $M_{\cS}^R$ is a tilting $RQ^{\cS}$-module by Lemma \ref{lem-comp-ptilt}, we have $RQ^{\cS} \in\thick M_{\cS_+}^R=\thick P_{\cS_+}^R$. We have an exact sequence $0\to\langle e\rangle\to RQ \to RQ^{\cS} \to 0$ such that $\langle e\rangle\in\proj RQ$ by Proposition \ref{prop-CB-thm1}(d). Since $\langle e\rangle\in\gen RQe$, we have $\langle e\rangle\in\add RQe$. Thus $RQ \in \thick P_{\cS}^R$ holds, as desired. (Alternatively it follows from Lemma \ref{idempotent reduction}.)

We prove the injectivity of $M_{-}^R : \Clus(Q) \to \stilt RQ$.
Take a morphism of rings $R \to k$ for a field $k$.
Let $\cS, \cS'\in\Clus(Q)$ with $\add M_{\cS}^R = \add M_{\cS'}^R$.
Then $\add M_{\cS}^k = \add M_{\cS'}^k$ holds.
Thus $\{\Dim M_{\alpha}^k \mid \alpha\in \cS_+\} = \{\Dim M_{\alpha}^k \mid \alpha\in \cS'_+\}$ and hence $\cS_+=\cS_+'$ holds by (a).
Since each cluster $\cS$ is determined by $\cS_+$, we have $\cS=\cS'$.
Thus $M_{-}^R$ is injective.
The proof of the other maps are similar.

For a principal ideal domain $R$, we prove the surjectivity of $M_{-}^R : \Clus(Q) \to \stilt RQ$.
For $T\in\stilt RQ$, by Proposition \ref{prop-CB-thm1}(c), there exists $\cS'\subset\Phi^{\rm rS}$ such that $M_{\cS'}^R \simeq T$.
Let $\cS=\cS' \sqcup \{-\ep_i \mid i \in Q_0, e_iT=0\}$.
Then $T$ is a tilting $RQ^{\cS}$-module.
By Lemma \ref{lem-comp-ptilt}, $\cS$ is a cluster satisfying $M_{\cS}^R\simeq T$.
The proof of the other maps are similar.
\end{proof}

The second step of the proof of Theorem \ref{thm-cluster-siltm} is the following observations.

\begin{lem}\label{lem-cluster-siltm-local}
Let $Q$ be a finite acyclic quiver, and $R$ a commutative Noetherian ring.
\begin{enumerate}[{\rm (a)}]
\item If $k$ is a field, then $M_{-}^k : \Clus(Q) \to \stilt kQ$ is an isomorphism of posets.
\item If $R$ is a local ring, then $P^R_{-}:\Clus(Q)\to\twosilt RQ$ is an isomorphism of posets.
\item For each prime ideals $\p\supseteq\q$ of $R$, $(-)_{\q} : \twosilt R_{\p}Q \to \twosilt R_{\q}Q$ is an isomorphism of posets.
\item If $R$ is ring indecomposable, then for any $\p\in\Spec R$, $(-)_{\p}: \twosilt RQ \to \twosilt R_{\p}Q$ is an isomorphism of posets.
\end{enumerate}
\end{lem}
\begin{proof}
(a) Since $k$ is a field, we have a bijection $M_{-}^k : \Clus(Q) \to \stilt kQ$ by Proposition \ref{prop-cluster-stilt}(b). It remains to show that, for $\cS,\cS'\in\Clus(Q)$, $\cS\ge\cS'$ if and only if $M_{\cS}^k\ge M_{\cS'}^k$, that is, $\gen M_{\cS}^k\supseteq\gen M_{\cS'}^k$. Since $M_{\cS}^k$ is a tilting $kQ^{\cS}$-module, $\gen M_{\cS}^k\supseteq\gen M_{\cS'}^k$ holds if and only if $M_{\cS'}^k$ is a $kQ^{\cS}$-module satisfying $\Ext^1_{kQ^{\cS}}(M_{\cS}^k,M_{\cS'}^k)=0$. These are nothing but the conditions (ii) and (i) in Definition \ref{dfn-cluster}(c).

(b) Let $(R,\m,k)$ be a local ring.
Consider the maps
\[
\Clus(Q) \xrightarrow{P^R_{-}} \twosilt RQ  \xrightarrow{k \otimes_{R}-} \twosilt kQ \xrightarrow{H^0} \stilt kQ.
\]
By \cite[Proposition 4.3]{IK} and since $RQ$ is a semi-perfect ring, $k\otimes_R-$ is an isomorphism of posets.
By \cite{AIR} and since the fact that support tilting $kQ$-modules are nothing but support $\tau$-tilting $kQ$-modules, $H^0$ is an isomorphism of posets.
Moreover the composite $H^0\circ(k\otimes_R-)\circ P^R_{-}=M_{-}^k:\Clus(Q)\to\stilt kQ$ is an isomorphism of posets by (a).
Thus the map $P^R_{-}:\Clus(Q)\to\twosilt RQ$ is an isomorphism of posets.

(c)
We have $P_{-}^{R_{\q}} = \Big[\Clus(Q) \xrightarrow{P_{-}^{R_{\p}}} \twosilt R_{\p}Q\xrightarrow{(-)_{\q}} \twosilt R_{\q}Q\Big]$.
Since both of $P_{-}^{R_{\p}}$ and $P_{-}^{R_{\q}}$ are isomorphisms of posets by (b), so is $(-)_{\q} : \twosilt R_{\p}Q \to \twosilt R_{\q}Q$.

(d)
Since $P_{-}^{R_{\p}} = \Big[\Clus(Q) \xrightarrow{P_{-}^{R}} \twosilt RQ\xrightarrow{(-)_{\p}} \twosilt R_{\q}Q\Big]$ is an isomorphism by (b), $(-)_{\p} : \twosilt RQ \to \twosilt R_{\p}Q$ is surjective.
Thanks to (c) and Proposition \ref{prop-local-iso}(a$^{\prime}$)$\Rightarrow$(a), $(-)_{\p} : \twosilt RQ \to \twosilt R_{\p}Q$ is an isomorphism of posets.
\end{proof}

\begin{proof}[Proof of Theorem \ref{thm-cluster-siltm}]
(a) Take any $\p\in\Spec R$. Since
\[P_{-}^{R_{\p}} = \Big[\Clus(Q) \xrightarrow{P_{-}^{R}} \twosilt RQ\xrightarrow{(-)_{\p}} \twosilt R_{\q}Q\Big]\]
and $(-)_{\p} : \twosilt RQ \to \twosilt R_{\p}Q$ are isomorphisms of posets by Lemma \ref{lem-cluster-siltm-local}(b) and (d), so is
$P_{-}^R:\Clus(Q)\to\twosilt RQ$.
Since the injective map $M_{-}^R:\Clus(Q)\to\stilt RQ$ given in Proposition \ref{prop-cluster-stilt} is a composition $H^0\circ P_{-}^R$, we have an injective map $H^0:\twosilt RQ\to\stilt RQ$. 
This is bijective by Proposition \ref{idempotent surjective} and Example \ref{RQ dense}, and an isomorphism of posets by \cite[Proposition A.6]{IK}.

This is bijective by Proposition \ref{idempotent surjective}, and an isomorphism of posets by \cite[Proposition A.6]{IK}.

(b) Thanks to Proposition \ref{prop-cluster-stilt}, it suffices to prove surjectivity.
Let $P\in\twopsilt RQ$. Then Bongartz completion \cite[Lemma 4.2]{IJY} gives $P'\in\sK^{\rm b}(\proj RQ)$ such that $P\oplus P'\in\twosilt RQ$.
By (a), there exists $\cS\in\Clus(Q)$ satisfying $\add (P\oplus P')=\add P^{R}_{\cS}$.
For each $\p\in\Spec R$, the category $\add P^{R_\p}_{\cS}$ is Krull-Schmidt since $\End_{\sD(R_\p Q)}(P^{R_\p}_\alpha)=R_\p$ is a local ring for each $\alpha\in\Cv(Q)$.
Thus there exists a subset $\cS^\p\subset\cS$ such that $\add P_\p=\add P^{R_\p}_{\cS^\p}$.

We claim that $\cS^\p$ is independent of $\p\in\Spec R$.
For each $\p\supseteq\q$ in $\Spec R$, applying $(-)_\q$ for $\add P_\p=\add P^{R_\p}_{\cS^\p}$, we have $\add P^{R_\q}_{\cS^\q}=\add P_\q=\add P^{R_\q}_{\cS^\p}$ and hence $\cS^{\q}=\cS^{\p}$. Since $R$ is ring indecomposable, the claim follows. Thus there exists $\cS'\subseteq\cS$ satisfying $\add P_\p=\add P^{R_\p}_{\cS'}=\add(P^R_{\cS'})_\p$ for each $\p\in\Spec R$.
By \cite[Proposition 2.26]{Iyama-Wemyss-maximal}, we have $\add P=\add P^R_{\cS'}$.
Thus $P_{-}^R:\pClus(Q)\to\twopsilt RQ$ is surjective.

For each $M\in\ptilt RQ$, let $P$ be a (2-term) projective resolution of $M$. Then $P\in\twopsilt RQ$ and $H^0(P)=M$ hold. There exists $\cS\in\pClus(Q)$ satisfying $\add P=\add P_{\cS}^R$. Then $\cS_+\in\pClus_+(Q)$ satisfies $\add M_{\cS_+}^R=\add H^0(P_{\cS}^R)=\add H^0(P)=\add M$. Thus $M_{-}^R:\pClus_+(Q)\to\ptilt RQ$ is surjective.
\end{proof}

\section{Torsion classes on Dynkin quivers}\label{section 4}
For a prime ideal $\p$ of $R$, let $\kappa(\p)=R_\p/\p R_\p$, $\L_\p=R_\p\otimes_R \L$, and we consider a finite dimensional $\kappa(\p)$-algebra $\kappa(\p)\otimes_R\L$.
By \cite[Theorem 3.16]{IK}, there exists a canonical injective map
\[
\Phi_{\rm t} : \tors\L \longrightarrow \bT_R(\L):=\prod_{\p\in\Spec R}\tors(\kappa(\p)\otimes_R\L), \quad \cT \mapsto (\kappa(\p)\otimes_R\cT)_{\p\in\Spec R},
\]
where $\kappa(\p)\otimes_R\cT:=\{\kappa(\p)\otimes_R X \mid X\in\cT\}$ is a torsion class of $\mod(\kappa(\p)\otimes_R\L)$.
One of the fundamental problems is to give an explicit description of the image of this map.
A necessary condition for elements in $\bT_R(\L)$ to belong to $\Im\Phi_{\rm t}$ is the following: 
For prime ideals $\p\supseteq \q$ of $R$, we have a map
\begin{eqnarray*}
&{\rm r}_{\p\q} : \tors(\kappa(\p)\otimes_R\L) \longrightarrow \tors(\kappa(\q)\otimes_R\L), \quad \cT \mapsto \kappa(\q)\otimes_{R_\p}\{X \in \mod\L_\p \mid \kappa(\p)\otimes_{R_\p}X \in \cT\}.&
\end{eqnarray*}
It is known that any element $(\cX^\p)_\p$ in $\Im\Phi_{\rm t}$ is \emph{compatible}, that is, ${\rm r}_{\p\q}(\cX^\p) \supseteq \cX^\q$ holds for any prime ideals $\p\supseteq \q$ of $R$ \cite[Proposition 3.20]{IK}.
We say that a Noetherian $R$-algebra $\L$ is \emph{compatible} if all compatible elements of $\bT_R(\L)$ belong to $\Im\Phi_{\rm t}$.
Then $\tors\L$ can be described as follows:
\begin{equation}\label{r description}
\tors\L \simeq \left\{(\cX^\p)_\p\in \bT_R(\L)  \,\middle|\, {\rm r}_{\p\q}(\cX^\p) \supseteq \cX^\q, {}^{\forall}\p \supseteq \q \in \Spec R \right\}.
\end{equation}

Now we prove our second main result of this paper. It was proved in \cite[Example 4.19]{IK} under the additional assumption that the ring $R$ contains a field.
\begin{thm}\label{thm-torsRQ}
Let $Q$ be a Dynkin quiver, and $R$ a commutative Noetherian ring.
Then $RQ$ is compatible, and we have isomorphisms of posets
\[
\tors RQ \simeq \Hom_{\rm poset}(\Spec R, \twosilt RQ)\simeq \Hom_{\rm poset}(\Spec R, \Clus(Q)).
\]
\end{thm}
\begin{proof}
By Theorem \ref{thm-cluster-siltm}, it suffices to prove the left isomorphism.
Clearly we can assume that $R$ is ring indecomposable (see the proof of \cite[Corollary 4.18]{IK}).
Thanks to Proposition \ref{thm-r-isom}, it suffices to show that $\L=RQ$ satisfies the conditions (a) and (b) there.

(a) This is immediate from by Lemma \ref{lem-cluster-siltm-local}(d).

(b) By \cite[Proposition 4.3]{IK} and \cite{AIR}, we have isomorphisms of posets
\[
\twosilt R_{\p}Q \xrightarrow[\sim]{\kappa(\p)\otimes_{R_\p}-} \twosilt (\kappa(\p) Q) \xrightarrow[\sim]{\gen} \ftors(\kappa(\p)Q) = \tors(\kappa(\p) Q),
\]
where the last equality holds since $Q$ is Dynkin. 
\end{proof}

\section*{Acknowledgements}
The first author is supported by JSPS Grant-in-Aid for Scientific Research (B) 22H01113, (B) 23K22384. The second author was supported by Grant-in-Aid for JSPS Fellows JP22J01155.

\end{document}